\documentclass[12pt]{amsart}
\usepackage{amsbsy,amssymb,amscd,amsmath,amsthm, hyperref,tikz,subcaption, pgf}
\usepackage{a4wide}
\usetikzlibrary{arrows}

\newcommand{\out}[1]{\ }

\newcommand{\DD}{{\mathbb D}}

\newcommand{\calc}{{\mathcal C}}

\newcommand{\RR}{{\mathbb R}}
\newcommand{\CC}{{\mathbb C}}

\renewcommand{\ge}{\geqslant}
\renewcommand{\le}{\leqslant}
\renewcommand{\phi}{\varphi}
\renewcommand{\epsilon}{\varepsilon}

\renewcommand{\Im}{\im}

\newcommand{\OUT}[1]{deleted}

\newcommand{\eps}{\varepsilon}

\DeclareMathOperator{\isdef}{\overset{def}{=}}

\DeclareMathOperator{\im}{Im}

\DeclareMathOperator{\psh}{PSH}

\DeclareMathOperator{\Capa}{Cap}

\makeatletter
\@namedef{subjclassname@2020}{\textup{2020} Mathematics Subject Classification}
\makeatother

\makeatletter
\newtheorem*{rep@theorem}{\rep@title}
\newcommand{\newreptheorem}[2]{%
\newenvironment{rep#1}[1]{%
 \def\rep@title{#2 \ref{##1}}%
 \begin{rep@theorem}}%
 {\end{rep@theorem}}}
\makeatother

\newtheorem{theorem}{Theorem}[section]
\newreptheorem{theorem}{Theorem}
\newtheorem{corollary}[theorem]{Corollary}
\newtheorem{lemma}[theorem]{Lemma}

\newtheorem{conjecture}[theorem]{Conjecture}

\theoremstyle{definition}
\newtheorem{definition}[theorem]{Definition}

\theoremstyle{remark}
\newtheorem{remark}[theorem]{Remark}

\numberwithin{equation}{section}

\newtheoremstyle{case}
{3pt}
  {3pt}
  {}
  {}
  {\bfseries}
  {:}
  {.5em}
  {}
\theoremstyle{case}

\numberwithin{subcase}{case}

\begin{document}
\title{Pluripolar hulls and fine holomorphy}
\dedicatory{To Jaap Korevaar on the occasion of his 100-th birthday,\\ in admiration and gratitude}
\author{Jan Wiegerinck}
\address{KdV Institute for Mathematics
\\University of Amsterdam
\\Science Park 105-107
\\P.O. box 94248, 1090 GE Amsterdam
\\The Netherlands}
\email{j.j.o.o.wiegerinck@uva.nl}
\thanks{Part of this work was done while the author visited Jagiellonian University in Krakow. He is grateful to Armen Edigarian, Lukasz Kosinski, and W\l odzimierz Zwonek for the invitation and for very stimulating discussions.}

\subjclass[2020]{30G12, 30A14, 31C10, 31C40, 32U05, 32U15}
\keywords{finely holomorphic function, pluripolar hull, fine analytic continuation}
\begin{abstract} 
Examples by Poletsky and the author and by Zwonek show the existence of nowhere extendable holomorphic functions with the property that the pluripolar hull of their graphs is much larger than the graph of the respective functions, and contains multiple sheets.  We will explain this phenomenon by fine analytic continuation of the function over part of a Cantor-type set involved. This gives more information on the hull, and  allows for weakening and effectiveness of the conditions in the original examples.
\end{abstract}
\maketitle
\section{Introduction}
We study the relation between pluripolar hulls of graphs of certain holomorphic functions and fine holomorphy. These notions are recalled in Section 2. Pluripolar hulls are of importance because of an old result of Zeriahi \cite{Ze} that roughly states that a pluripolar set $X\subset\CC^n$ is {\em complete} if it is an $F_\sigma$ as well as $G_\delta$ and equal to its pluripolar hull. The last two conditions are necessary too. Recall that $X\subset\CC^n$ is pluripolar if it is contained in the $-\infty$ set of a plurisubharmonic function, and it is complete if it equals the $-\infty$ set of such a function.

 Levenberg, Martin, and Poletsky in \cite{LMP} conjectured that the pluripolar hull of a graph of an analytic function $f$ is equal to the graph of the maximal analytic extension of $f$ in the sense of Weierstrass. This conjecture was refuted in \cite{EdW1}. Other counterexamples to this conjecture showed remarkable behavior: the hull could extend far beyond the domain of the maximal extension and could consist of many sheets, \cite{ESZ,PW,Z}. The connection with fine holomorphy was made in \cite{EJ}. This eventually led to Theorem \ref{thmEW}, which together with the observation in \cite{EEW} that the examples from \cite{EdW1,EdW3} are all instances of finely analytic continuation, suggested the following refined conjecture.
 
 \begin{conjecture} \cite{JW12} The pluripolar hull of a graph of an analytic function $f$ consists of the graph of the (possibly multi-valued) maximal finely analytic extension of $f$. 
 \end{conjecture}
 
In \cite{JW12} the author boldly stated without much proof that in all the known examples it can be seen that the pluripolar hull occurs because of finely analytic continuation. The purpose of this paper is to corroborate this statement and the conjecture. 
 
 \medbreak
Let $\calc=\calc(a,b)$, $a=(a_0, a_1,a_2,\cdots)$, $b=(b_0, b_1,b_2,\cdots)$, $a_j<b_j$, be a Cantor set in the real line formed by omitting  from the interval $[a_0,b_0]$ open intervals $(a_i,b_i)$, $i\ge 1$, in such a way that $[a_i,b_i]\cap [a_j,b_j]=\varnothing$ ($i > j\ge 0$) and the maximal length of an interval in $[a_0,b_0]\setminus\bigcup_{j=1}^N(a_j, b_j)$ tends to 0 for $N\to\infty$. Without loss of generality, we can and will assume that for all $i$ we have $b_{i+1}-a_{i+1}\le b_i-a_i$. 

In \cite{PW} Poletsky and the author studied the holomorphic function $f_{a,b}$ defined by
\[f_{a,b}(z)=\frac{z-b_0}{z-a_0}\prod_{i=1}^\infty \frac{z-a_i}{z-b_i} \]
and its  square roots, which are both well-defined single valued functions on the complement of $\calc(a,b)$. The main result here was that it is possible to choose the Cantor set in such a way that the \emph{pluripolar hull} in $\CC^2$ of the graph $\Gamma$ of each of these roots consists of precisely the graphs of the two square roots above the complement of the Cantor set. What happens above the Cantor set remained unclear. Then Zwonek \cite{Z} constructed holomorphic functions on the unit disc that can not be extended over the disc, but with the property that the pluripolar hull of their graph consists of infinitely many points above $\CC\setminus \{|z|=1\}$, see also Siciak, \cite{S}. In \cite{ESZ} similar results were obtained for functions on the complement of certain Cantor-like sets and, moreover, the hull above these Cantor sets was partially described.

We show that a computable condition on Cantor sets $\calc(a,b)$ in terms of properties of the sequences $a$ and $b$ guarantees that the pluripolar hull of a square root of $f_{a,b}$ 
indeed comes about because of fine analytic extension of this root over part of the Cantor set. This is done in Section 4.  In particular, it follows, as in \cite{ESZ}, that the pluripolar hull also contains points above this part of the  Cantor set. Under some extra conditions we show  in Section 5 that above this part the hull consists of precisely the graphs of the fine extension of the roots.  In Section 3 we deal with some elementary properties of $f_{a,b}$ and its square roots, and correct an error in \cite{PW}. 

In the final Section 6 we follow up on the examples in \cite{Z} and \cite{S}. In these articles Blaschke products $B$ with extra conditions on their zeroset are constructed. The domain of existence of $B$ is the unit disc and for every holomorphic function $f$ on a larger domain than the unit disc, the hull of the graph of $fB$ is strictly larger than the graph of $fB$ over the unit disc and may contain infinitely many sheets. We show that the Blaschke product $B$ and therefore the functions $fB$ admit finely analytic continuation over part of the unit circle. As in \cite{ESZ} the condition is local, and altogether less restrictive than the ones in \cite{Z,S}. Moreover, we obtain information on the hull over this part of the unit circle.

\section{Finely holomorphic functions and pluripolar hulls}

Recall that the fine topology on $\RR^n$, in particular on $\RR^2=\CC$ is the coarsest topology that makes all subharmonic functions continuous. See \cite{Fu1, Fu2, Doob, AG} for background on the fine topology. On finely open sets one can define \emph{finely holomorphic functions}. For our purposes the following definition is the most convenient. See \cite{Fu3, Fu4, FGW} for more information on finely holomorphic functions. 
\begin{definition}
Let $\Omega$ be  finely open subset of $\CC$. A function $f:\Omega\to \CC$ is called \emph{finely holomorphic} (or \emph{finely analytic}) on $\Omega$ if for every $z\in\Omega$ there exists a fine neighborhood $K_z\subset\Omega$, which can be taken to be compact in the Euclidean topology, with the property that $f|_{K_z}$ is a uniform limit of rational functions with poles off $K_z$.
\end{definition}
Note that if $\Omega$ is Euclidean open, then finely holomorphic functions on $\Omega$ are just ordinary holomorphic functions. Finely open sets of the form $\{z\in D: h(z)>c\}$, where $h$ is a subharmonic function on an open set $D$ in $\CC$, form a base of the fine topology.

Next let us recall that the ($\CC^n$-)pluripolar hull of a pluripolar set  $X\subset \CC^n$ is defined as
\[X^*_{\CC^n}\isdef\{z\in\CC^n:\ \forall u\in\psh(\CC^n)\ u|_X=-\infty\Rightarrow u(z)=-\infty\}.\]
Here $\psh(\CC^n)$ denotes the plurisubharmonic functions on $\CC^n$. See \cite{LP} for background and properties of pluripolar hulls. 

Note that if $X$ is a sufficiently large part of the graph of a holomorphic function $f$, (the projection on the domain of $f$ should have positive capacity), then $X^*_{\CC^n}$ contains the graph of any analytic extension of $f$.

Edlund and J\"oricke, \cite{EJ} observed a connection between pluripolar hulls of graphs and fine analytic extension.
This eventually led to the following result \cite[Theorem 4.5]{EW2}.
\begin{theorem}
\label{thmEW} Let $E$ be a set of positive capacity in $\CC$ and let $f$ be a finely holomorphic function on a fine neighborhood of $E$. Let $X$ denote the graph of $f$ over $E$. Then $X^*_{\CC^n}$ contains the graph of the (possibly multi-valued) maximal finely analytic continuation of $f$.
\end{theorem}

\section{The product $f_{a,b}$ and its square roots}
We denote by $d(z,A)$ the Euclidean distance from $z\in \CC$ to $A\subset \CC$. Writing
\[f_{a,b}(z)=\left(1+\frac{a_0-b_0}{z-a_0}\right)\prod_{i=1}^\infty \left(1+\frac{b_i-a_i}{z-b_i}\right)\]
we see that the product converges uniformly on sets \[\Omega_\delta=\{z: d(z,\calc(a,b))\ge\delta\}, \quad \delta>0,\]
because $0\le \sum_{i=1}^\infty (b_i-a_i)\le b_0-a_0$.
Hence $f_{a,b}$ represents a holomorphic function on the complement of $\calc(a,b)$. It is bounded at infinity and therefore extends to be holomorphic at infinity. The Laurent expansion on $\{|z|>|a_0|+|b_0|\}$ is 
\begin{equation}\label{eq1}
    f_{a,b}(z)=1+ \frac{c_1}{z}+\frac{c_2}{z^2}+\cdots,
\end{equation}
with $c_1= -(b_0-a_0)+\sum_{i=1}^\infty (b_i-a_i) $, so that $c_1$ equals the length (1-dimensional Lebesgue measure) of the Cantor set. 

If $\calc(a,b)$ has length zero it was claimed in \cite{PW} that $f_{a,b}$ is constant. This is in general not true. 
Consider $\log f_{a,b}$ a branch of the logarithm, which is well-defined in a neighborhood of $\infty$. Its derivative equals \[\left(\frac1{z-b_0}-\frac1{z-a_0}\right)-\sum_{i=1}^\infty\left(\frac1{z-b_i}-\frac1{z-a_i}\right)=\frac 1z \sum_{k=1}^\infty \frac{b_0^k-a_0^k-\sum_{j=1}^\infty b_j^k-a_j^k}{z^k}.\] 
If $\calc(a,b)$ has length zero, then the first term ($k=1$) in this sum vanishes, but the others generally  need not.

Let
\[f^N=f^N_{a,b}=\frac{z-b_0}{z-a_0}\prod_{i=1}^N \frac{z-a_i}{z-b_i} \]
be a partial product. There exist two holomorphic branches $g^N_+$ and $g^N_-$  of $\sqrt{f^N}$ on 
\[(\CC\setminus [a_0,b_0])\cup\left(  \bigcup_{j=1}^N (a_j, b_j)\right)\] with $\lim\limits_{z\to\infty} g^N_+(z)=1$ and $\lim\limits_{z\to\infty} g^N_-(z)=-1$. 

On 
\[(\CC\setminus\RR)\cup (a_0,b_0)\setminus \left(  \bigcup_{j=1}^N [a_j, b_j]\right)\] 
there are two branches  of $\sqrt{f^N}$: $\tilde g^N_+$ with $\lim\limits_{z\to\infty \atop \Im z>0}\tilde g^N_+(z)=1$, and $\tilde g^N_-$ with $\lim\limits_{z\to\infty \atop \Im z>0}\tilde g^N_-(z)=-1$.

In particular, it follows that $\lim\limits_{z\to\infty \atop \Im z<0}\tilde g^N_\pm(z)=\mp 1$, as $\tilde g^N_{\pm}=g^N_\mp$ on the lower half-plane.
The four functions $g^N_\pm$, $\tilde g^N_{\pm}$ are analytic extensions of each other and together form a 2-valued holomorphic function on $\CC\setminus\{ a_j,b_j, j=0\ldots N\}$.

Clearly $g^N_{\pm}$ both converge uniformly on compact sets to holomorphic square roots $g_{\pm}$ of $f_{a,b}$ on 
\begin{equation} \label{D}D=(\CC\setminus [a_0,b_0])\cup\left(  \bigcup_{j=1}^\infty (a_j, b_j)\right),\end{equation} 
while $\tilde g^N_{\pm}$ both converge uniformly on compact sets to square roots $\tilde g_{\pm}$ of $f_{a,b}$ on 
\begin{equation}\label{tildeD}\tilde D=\CC\setminus \RR.
\end{equation}

Obviously $g_{\pm}=\tilde g_{\pm}$ on the upper half-plane and $g_{\pm}=-\tilde g_{\pm}$ on the lower half-plane. The maximal analytic extension of each of these functions is a single valued function on $D$.

\section{Fine holomorphy on the Cantor set}
We start with a condition on $\calc(a,b)$ that ensures that $f_{a,b}$ converges uniformly on compact sets in  a set $E$ that contains a subset of $\calc(a,b)$ and is sufficiently big, namely that it is a fine neighborhood of certain points in $\calc(a,b)$. To make the connection with \cite{PW} we will write $b_j-a_j=e^{-jc_j}$, where $\{c_j\}_j$ is an increasing sequence of positive numbers tending to $\infty$.

The following lemma is obvious. 
\begin{lemma}\label{lem4.1} Let $(p_n)_n$ 
be a sequence of positive numbers with converging sum. 
 Then \begin{equation}\label{eq1a}
    \sum_{n=1}^\infty \frac{b_n-a_n}{|z-b_n|} \quad \text{ and } \quad \left(1+\frac{a_0-b_0}{z-a_0}\right)\prod_{i=1}^\infty \left(1+\frac{b_i-a_i}{z-b_i}\right)
\end{equation} converge uniformly on compact sets in 
\begin{equation}\label{eq4.1a} E_N=\bigcap_{n=N}^\infty\{z: b_n-a_n\le p_n |z-b_n|\}\setminus \{b_1,\ldots, b_{N-1}\},\quad N>0.
\end{equation} 
    \end{lemma}
Note that only if $\sum \frac{e^{-jc_j}}{p_j}$ is finite, $E_N\cap \calc(a,b)$ will eventually be nonempty.

The following Lemma is a variation of Lemma 2 and Example 14 in \cite{ESZ}.
\begin{lemma}\label{Lem4.2}  Let $\calc(a,b)$ be a Cantor set formed by sequences $a$, $b$ as before and let $\{c_j\}_j$ be an increasing sequence tending to $\infty$ with $b_j-a_j=e^{-jc_j}$, and such that
\begin{equation}
     \sum_{j=1}^\infty\frac{1}{jc_j}<\frac12.\label{eq4.2}
\end{equation}

Then for $N$ sufficiently large, the set $E_N$ in \eqref{eq4.1a} contains a fine neighborhood $E$ of a nonempty subset of $C(a,b)$.The functions in \eqref{eq1a} converge uniformly on compact subsets of $E$.
\end{lemma}
 
\begin{proof} We can assume $b_0-a_0=1$. Let 
 $p_j=e^{-jc_j/2}$ and set

\begin{equation}
F_N=\bigcup_{i=1}^\infty[a_i,b_i]\cup\bigcup_{n=N}^\infty\{z\in\CC:\ |z-b_n|\le(b_n-a_n)/p_n\}. 
\end{equation}
Note that $F_N$ is compact.

We have $\Capa[a_j,b_j]=e^{-jc_j}/4$ and $\Capa(\{|z-b_j|\le (b_j-a_j)/p_j\})= e^{-jc_j/2}$.
We now use Theorem 5.1.4 in \cite{Ra}, which estimates the capacity of a union of Borel sets. It reads in our setting
\begin{equation}
\begin{split}
    \frac{1}{\log(\frac1{\Capa(F_N)})}&\le \sum_{j=1}^\infty\frac{1}{\log(4e^{jc_j})}+\sum_{j=N}^\infty \frac{1}{\log(e^{jc_j/2})}\\
    &=\sum_{j=1}^\infty\frac{1}{jc_j+\log 4}+\sum_{j=N}^\infty \frac{2}{jc_j}.
    \end{split}
\end{equation}

Let $\eps>0$. By taking $N$ sufficiently large, the final sum $\sum_{j=N}^\infty$ in this equation can be made smaller than $\eps$. Then
\begin{equation}
    \frac{1}{\log(\frac1{\Capa(F_n)})}\le \sum_{j=1}^\infty \frac1{jc_j} +\eps,
\end{equation}
or 
\begin{equation}
    \Capa(F_N)\le \exp\left(\frac{-1}{\sum_{j=1}^\infty \frac1{jc_j}+\eps}\right)<\frac14,
\end{equation}
 if $\sum \frac1{jc_j} \le 1/2$ and $\eps$ is sufficiently small.

Define the compact set $J_N:=[0,1]\cup F_N$. Let
$$u(z):= g_{F_N}-g_{J_N},$$
where $g_X$ denotes the Green function of $X$ with pole at infinity. Being the difference of two subharmonic functions, $u$ is finely continuous, therefore $\tilde E=\{u>0\}$ is a finely open set. 
Since the complement of $F_N$ is contained in $E_N\cup\{b_1,\ldots, b_{N-1}\}$, and $u\equiv0$ on $F_N$, the functions in \eqref{eq1a} converge uniformly on compact sets in the finely open set  $E=\tilde E\setminus \{b_1,\ldots, b_{N-1}\}$.

It remains to show that $E$ contains points of $\calc(a,b)$. We have
$$u(\infty)=\log (\Capa J_N) -\log (\Capa F_N),$$
cf. \cite[Thm 5.2.1]{Ra}. Since $J_N$ contains $[0,1]$, $\Capa J_N\ge 1/4$, hence $u(\infty)>0$ if $N$ is sufficiently large.
Let $\varphi$ be the Riemann map of the unit disc to the complement $\Omega$ of $J_N$  that maps $0$ to $\infty$. Observe that the boundary of $\Omega$ is locally connected and regular for the Dirichlet problem, therefore $\varphi$ extends continuously to the unit circle and $G_{J_N}$ is a continuous function with $G_{J_N}\equiv 0$ on $J_N$. Note that for every $\theta$ the curve $r\mapsto \varphi(re^{i\theta})$, ($0\le r\le 1$) is non-thin at  $\varphi (e^{i\theta}) \in J_N$. 

Let $v=u\circ\varphi$. Then $v$ is harmonic on the unit disc and $v(0)=u(\infty)>0$. Since
$$v(0)=\frac 1{2\pi} \int_0^{2\pi}v^*(e^{i\theta})\, d\theta,$$
where $v^*(e^{i\theta})=\lim_{r\uparrow 1} v(re^{i\theta})$ a.e.,
it follow that  $\lim_{r\uparrow 1} v(re^{i\theta})>0$ on a subset $\Gamma$ of positive measure on the unit circle. 
 Using the above observations, we find for $\theta\in \Gamma$
$$u\circ\varphi(e^{i\theta})= g_{F_N}(\varphi(e^{i\theta}))- g_{J_N}(\varphi(e^{i\theta}))=\liminf_{r\uparrow 1}g_{F_N}(\varphi(re^{i\theta}))- g_{J_N}(\varphi(re^{i\theta}))=\lim_{r\uparrow 1}v(re^{i\theta})>0.$$
It follows that $\varphi(e^{i\theta})\in J_N\setminus F_N\subset \calc(a,b)$ and we are done.

\end{proof}
\begin{remark} It is a well-known consequence of Wiener's criterion that  relatively finely open subsets of $[0,1]$ have positive Lebesgue measure. We don't need this fact, but it strengthens the connection with \cite[Lemma 2]{ESZ}.
\end{remark}

\begin{theorem}\label{thm2}Assume that $\calc(a,b)$ satisfies the conditions of Lemma \ref{Lem4.2}. Let $E\subset\calc(a,b)$ be the set constructed in the lemma. The functions $f_{a,b}$ and $\tilde g_\pm$ extend over $E$ as finely holomorphic functions from the upper half plane to the lower half plane.
\end{theorem}
\begin{proof}By Lemma \ref{Lem4.2} the functions $f^N$, and therefore the functions $g^N_{\pm}$ converge uniformly on compact sets in a finely open neighborhood $E$ of points of $\calc(a,b)$ to $f_{a,b}$, respectively $\tilde g_{\pm}$. By the definition it follows that their limits are finely holomorphic on  $E$. These limits clearly coincide with $f_{a,b}$, respectively $\tilde g_{\pm}$ on $E\cap(\CC\setminus\RR)$ and are therefore the sought for finely holomorphic extensions.
\end{proof}

\begin{remark} One may wonder if, perhaps, $f_{a,b}$ extends as a finely holomorphic function to the complement of the set of poles $\{b_j\}$. However, this is not the case. The polar set $\{b_j\}$ is of the first Baire category in its closure $\calc(a,b)$ and then Theorem 3.4. in \cite{FGW} states that $f_{a,b}$ would extend finely holomorphically over part of $\{b_j\}$, which is clearly not the case.
\end{remark}

\begin{corollary}\label{cor1} With $D$ and $\tilde D$ as in \eqref{D}, \eqref{tildeD} and assuming that $\calc(a,b)$ satisfies the conditions of Lemma \ref{Lem4.2}, the four finely analytic function elements $\tilde g_{\pm}$  on $\tilde D\cup E$  and $g_{\pm}$  on $D$ are (finely) holomorphic continuations of each other. 
Let $X$ denote the graph of $g_+$, then $X^*_{\CC^2}$ contains the graphs of $\tilde g_{\pm}$ over $\tilde D\cup E$ and the graph of $g_-$ over $D$. In particular $X^*_{\CC^2}$ consists of precisely two points over $\CC\setminus \calc(a,b)$ and at least two points over $E$.
\end{corollary}
\begin{proof}
Note  that $\tilde D\cup E$ is finely connected, and $D$ is (finely) connected. Hence the first statement follows directly from Theorem \ref{thm2} and the fact that $g_{\pm}=\tilde g_{\pm}$ on the upper half plane and $g_{\pm}=-\tilde g_{\pm}$ on the lower half plane. 
That $X^*_{\CC^2}$ consists of at most two points $(z,w)$ for $z\in \CC\setminus \calc(a,b)$ is part of Theorem 3.6. in \cite{PW}. The proof of this part only uses that $\{c_n\}_n$ tends to $\infty$.
Theorem 4.5 in \cite{EW2} shows that points of the form $(z,\tilde g_{\pm}(z))$ belong to $X^*_{\CC^2}$ if $z\in E$. 
\end{proof}

\section{The pluripolar hull above $E$}
As a proof of principle we will show in this section that under a strong condition on the size of the deleted intervals,
 for every $z\in E$ the fiber $X^*_{\CC^2}\cap (\{z\}\times\CC)$ equals $\{(z, w):\   w^2=f_{a,b}(z)\}$. Here we have written $f_{a,b}$ also for its finely analytic extension over $E$.

\begin{lemma}\label{fa} Let $\calc(a,b)$ be a Cantor set as before, with $b_j-a_j=e^{-jc_j}$ and $c_j=(j+2)!$, in particular the condition of Lemma \ref{Lem4.2} is satisfied. Let $E$ and $F_N$
 be the sets constructed in Lemma \ref{Lem4.2} and $f=f_{a,b}$ the finely holomorphic function on $V=(\CC\setminus\calc(a,b))\cup E$. Then for $N$ sufficiently large
there is a plurisubharmonic function $v$ on $\CC^{2}$ such that
$\{v=-\infty\}\cap([a_0,b_0]\setminus F_N\times\CC)=\Gamma_f\cap([a_0,b_0]\setminus F_N\times\CC)$. 
\end{lemma}
\begin{proof} Without loss of generality we can assume that $[a_0,b_0]=[0,1]$. Set $p_j=e^{-jc_j/2}$. Let 
\[f_n(z)=\frac{z-b_0}{z-a_0}\prod_{i=1}^n \frac{z-a_i}{z-b_i}=\frac{P_n(z)}{Q_n(z)}, \]
where $P_n$ and $Q_n$ are monic polynomials of degree $n+1$.
We have for $C>1$
\begin{align}
    |Q_n(z)|\le (C+1)^{n+1} \quad \text{if $|z|<C$},\\
    \intertext{for $z\in [a_0,b_0]\setminus F_N$ we have}
    |P_n(z)|<1,\quad |Q_n(z)|<1,\\
    \intertext{and}
    |Q_n(z)|\ge \prod_{j=1}^n \frac{e^{-jc_j}}{p_j}=\prod_{j=1}^n e^{-j c_j/2} .
\end{align}
Thus on $[a_0,b_0]\setminus F_N$ we have the following estimates
\begin{equation}
    \begin{split}
        |f(z)-f_n(z)||Q_n(z)|&\le \left|\left(\prod_{j=n+1}^\infty \left(1+\frac{b_i-a_i}{z-b_i}\right)-1\right)P_n(z)\right|\\
        &\le \prod_{j=n+1}^\infty(1+p_j)-1\le \exp(\sum_{j=n+1}^\infty p_j)-1\le 2 p_{n+1},
    \end{split}
\end{equation}
    if $n\ge 2$. 
    
    Next for every $\delta>0$ and $|w-f(z)|\ge\delta$
    \begin{equation}
     |w-f_n(z)||Q_n(z)|\ge \delta \prod_{j=1}^n e^{-jc_j/2}.
\end{equation}
Now consider the plurisubharmonic functions \[v_n(z,w)= \log(|w-f_n(z)||Q_n(z)|)\quad\text{ on $\CC^2$.}\] Note that
\begin{equation}
    \frac{(n+1)c_{n+1}}{n c_n}=\frac{(n+1)(n+2)}{n} \to \infty, \quad \text{as $n\to \infty$.}
\end{equation} Let $\{e_n\}_n$ be a sequence of positive numbers such that $\sum e_n$ is finite but $\sum e_n \frac{-(n+1)c_{n+1}}{n c_n}=-\infty$, and form
\begin{equation}
    v(z,w)=\sum_{n=1}^\infty\frac{\max \{v_n(z,w),\log p_{n+1}\}}{nc_n}e_n.
\end{equation}
Because $v_n(z,w)\le (n+2)\log(C+1)$ on $\{|z|,|w|<C\}$, for $C$ sufficiently large, the function  $v(z,w)$ is plurisubharmonic on $\CC^2$.
For $z\in \calc(a,b)\setminus F_N$ the function $v(z,w)$ has the following properties.
If $w=f(z)$, then $v_n(z,w)\le \log 2 +\log p_{n+1}$ and therefore $v(z, f(z))=-\infty$. On the other hand, if $|w-f(z)|\ge \delta$, then there exists $n_0>0$ with 
\begin{equation}
    v_n(z,w)\ge \log(\delta) +\sum_{j=1}^n(-jc_j/2)\ge \log\delta -n c_n, 
\end{equation}
therefore $v(z,w)$ is finite.

\end{proof}

\begin{theorem}\label{thm5.3}Let $\calc(a,b)$ be a Cantor set that satisfies the conditions of Lemma \ref{fa}.
Let $X$ denote the graph of $g_+$ over $\CC\setminus \calc(a,b)$ and let $g$ denote the maximal finely analytic extension of $g_+$. Then $X^*_{\CC^2}\cap \{(D\cup E)\times \CC\}=\Gamma_{g}\cap \{(D\cup E)\times \CC\}$.
\end{theorem}
\begin{proof}
Because of Corollary \ref{cor1}, we only have to prove that for $z\in E$ $X^*_{\CC^2}\cap (\{z\}\times \CC= \{(z, g_+(z)), (z,g_-(z))\}$. 
The set $E$ has positive length and hence positive capacity. The function $v$  constructed in the previous lemma is plurisubharmonic function and equals $-\infty$ on $\{(z,f_{a,b} (z)), z\in E\}$. Then by Theorem \ref{thmEW} it equals $-\infty$ on the graph of $f_{a,b}$. 
Consider the function $u(z,w)=v(z,w^2)$. The function $u$ is plurisubharmonic and equals $-\infty$ on the graph of $g_+$ and for $z\in E$ we have $u(z,w)=-\infty$ if and only if $w$ is one of the two values of $g(z)$. 
\end{proof}

\section{Blaschke products}
In this section we will consider Blaschke products of the form
\begin{equation}\label{Bl1}
B(z)=z^l\prod_{j=1}^\infty\left(\frac{|a_j|}{a_j}\frac{a_j-z}{1-\overline{a_j}z}\right)=z^l\prod_{j=1}^\infty\left(\frac{1}{|a_j|}\frac{a_j-z}{1/\overline{a_j}-z}\right)
    \end{equation}
    on the unit disc $\DD$.
In the next lemma we formulate a condition similar to \eqref{eq4.2} that guarantees that such a Blaschke product extends as a finely holomorphic function over a part $E$ of the boundary of $\DD$, while $E$ is in the closure of the zeros of $B$. The lemma leads then to an easy proof of an extension of Theorem 1 in \cite{Z}.

It is well known that under the Blaschke condition $\sum_k(1-|a_k|)<\infty$ the product \eqref{Bl1} represents a function that is holomorphic on $\DD\cup (\DD^e\setminus X)$, where $\DD^e=\{z: |z|>1\}$ and $X=\{1/\overline{a_j}\}$, the set of poles of $B$.
\begin{lemma}\label{LemBl1} Let $S=\{e^{i\phi}:\ \alpha\le\phi\le\beta\}$ be a proper segment of $\partial \DD$, $(a_j)_j$ a sequence in $\DD$ that satisfies the Blaschke condition and such that $S\subset \overline{\{a_j,j=1,2, \ldots\}}$. Let $B$ be the Blaschke product with zeros $\{a_j\}$. Suppose that there exists an increasing sequence of positive numbers $\{c_j\}$ such that $\sum_{j=1}^\infty\frac{1}{jc_j}$ is finite and that
\begin{equation}\label{eqBl1}\left|a_j-\frac1{\overline{a_j}}\right|=e^{-jc_j} .\end{equation}
Then there exists a finely open set $E\subset \CC$ that meets $S$ with the property that $B$ extends as a finely holomorphic function over $E$. 
\end{lemma}
\begin{proof} 
The proof goes much the same as the proof of Lemma \ref{Lem4.2}. First observe that \[\left|\frac{1}{|a_j|}\frac{a_j-z}{1/\overline{a_j}-z}-1\right|=\left|\frac{1}{|a_j|}\frac{a_j-1/\overline{a_j}}{1/\overline{a_j}-z}-(1-1/|a_j|)\right|\le \frac{1}{|a_j|}\frac{|a_j-1/\overline{a_j}|}{|1/\overline{a_j}-z|}+\frac{(1-|a_j|)}{|a_j|}.\]
It follows that the product \eqref{Bl1} converges uniformly on compact sets in $$E_N=\left(\bigcap_{j=N}^\infty \{z: |1/\overline{a_j}-z|\ge |a_j-1/\overline{a_j}|/e^{-jc_j/2}\}\right)\setminus\{1/\overline{a_1},1/\overline{a_2},\ldots, 1/\overline{a_{N-1}}\}.$$

Let $$F_N=\bigcup_{j=N}^\infty \{z\in \CC:  |z-1/\overline{a_j}|\le e^{-jc_j/2}\}.$$
Put $J_N=S\cup F_N$ Note that  $J_N$  is compact, its complement in $\CC\cup\{\infty\}$ is simply connected and $\Capa J_N\ge \Capa S=\sin((\beta-\alpha)/4)$.  As in the proof of Lemma \ref{Lem4.2} we can take  $N$ so large that $\Capa F_N<\Capa J_N$. Now form again $u(z):=g_{F_N}- g_{J_N}$ and the proof continues just like the proof of Lemma \ref{Lem4.2}.
\end{proof}{}
\begin{theorem}\label{thmBl1} Let $S=\{e^{i\phi}:\ \alpha\le\phi\le\beta\}$ be a proper segment in $\partial \DD$ and let $B$ be a Blaschke product with zero set $\{a_k, b_l\}$, where $\alpha\le \arg a_k\le \beta$ and $\arg b_k\notin [\alpha,\beta]$. Suppose that the $a_k$ satisfy \eqref{eqBl1}. 

Let $U$ be a domain with $S\subset U$ and $E\subset S$ the set constructed in Lemma \ref{LemBl1}. For $f$ holomorphic on $U$ put $g=fB$ on $(U\setminus (\partial\DD\cup X))\cup E$. 
Then 
\begin{itemize}
    \item \[\Gamma^*_{g|_{\DD\cap U}}\supset\Gamma_{g|_{(U\setminus (S\cup X))\cup E}}\]
    \item \[\Gamma^*_{g|_{\DD^e\cap U}}\supset\Gamma_{g|_{(U\setminus (S\cup X))\cup E}}\]
\end{itemize}
In fact both $\Gamma^*_{g|_{\DD\cap U}}$ and $\Gamma^*_{g|_{\DD^e\cap U}}$ contain the graph of the maximal finely analytic extension of $g$.
\end{theorem}
\begin{proof} The function $f$ is holomorphic on $U$ and $B$ is finely holomorphic on $\CC\setminus (\partial \DD\cup X)\cup E$. therefore $g$ is finely holomorphic on the fine domain $(U\setminus(S\cup X))\cup E$. Theorem \ref{thmEW} now gives the result.
\end{proof}
\begin{corollary}\label{corBl1}
Let $B$ be a Blaschke product as in Theorem \ref{thmBl1} with the additional property that it admits no holomorphic extension outside $\DD$. Take $f(z)=p.v.\log(z+2)$ on $|z|<2$. Then of $g=fB$ cannot be extended outside $\DD$ but $\Gamma^*_{g|_{\DD\cap U}}$ contains the graph of the maximal finely holomorphic extension of $g$. In particular, all fibers $\Gamma^*_{g|_{\DD\cap U}}\cap (\{z\}\times \CC)$ will be infinite, and most of them will be polar sets in the complex line $\{z\}\times\CC$.
\end{corollary}
\begin{remark}
\begin{enumerate}
    \item In \cite{Z} Zwonek proves his version of Corollary \ref{corBl1} under the following condition on $B$:
\begin{equation}
   \lim_{r\to\infty} \Capa(\{ z\in\DD : |B(z)|\le r\})=0.\label{Z1}
\end{equation}
Siciak \cite{S} replaced this condition by the seemingly weaker condition: There exists $r_0$ with $0<r_0<1$ such that 
\[\Capa(\{ z\in\DD : |B(z)|\le r_0\})<1.\]
In fact it is not known whether Siciak's condition is really strictly weaker. In Proposition 7 of \cite{Z} two conditions are given that together guarantee \eqref{Z1}. Our condition is equivalent to Zwonek's second condition; we don't need his first.
\item 
In both \cite{Z} and \cite{S} the conditions on the capacity and the zeros of the Blaschke product are on all of $\DD$, whereas in Theorem \ref{thmBl1} and its corollary the condition on the zeros is only required in a sector. 
\item In \cite{ESZ} similar results as in Theorem \ref{thmBl1} and its corollary are proved by potential theoretic methods. The explanation by fine holomorphy is new.
\end{enumerate}
\end{remark}

\begin{remark}
While it would be possible to obtain more precise information of the pluripolar hull of the graph of the Blaschke product along the lines of Lemma \ref{fa} and Theorem \ref{thm5.3}, it is not clear how to do this in general for the products $fB$. We will not pursue this further. \end{remark}{}

\end{document}